\documentclass[oneside, 12pt]{amsart}

\usepackage[utf8]{inputenc}
\usepackage[margin=1in]{geometry}
\usepackage{amsthm}
\usepackage{enumitem}  
\usepackage[OT2,T1]{fontenc}
\usepackage{amscd, amssymb, enumitem, amsmath,mathrsfs, amsfonts, amsaddr}
\usepackage{url}
\usepackage{tikz}
\usepackage{fullpage}
\usepackage{microtype}
\usetikzlibrary{decorations.pathreplacing}
\usepackage[backref=page]{hyperref}
\usepackage{hyperref}
\usepackage[utf8]{inputenc}
\usepackage[main=english,russian]{babel}

\DeclareSymbolFont{cyrletters}{OT2}{wncyr}{m}{n}
\DeclareMathSymbol{\Sha}{\mathalpha}{cyrletters}{"58}

\usepackage[utf8]{inputenc}

\newtheorem{theorem}{Theorem}[section]

\newtheorem*{proposition*}{Proposition}

\newtheorem*{questiona*}{Question A}
\newtheorem*{questionb*}{Question B}
\newtheorem*{theorem*}{Theorem}
\newtheorem*{question*}{Question}

\theoremstyle{definition}

\newtheorem{example}[theorem]{Example}

\newtheorem{claim}[theorem]{Claim}
\newtheorem{remark}[theorem]{Remark}

\theoremstyle{remark}

\title{Low rank specializations of elliptic surfaces}
\author{Mentzelos Melistas}

\address{University of Twente, Department of Applied Mathematics, Drienerlolaan 5, 7522 NB Enschede, The Netherlands}

\date{\today}

\begin{document}

\maketitle

\begin{abstract}
    Let $E/\mathbb{Q}(T)$ be a non-isotrivial elliptic curve of rank $r$. A theorem due to Silverman implies that the rank $r_t$ of the specialization $E_t/\mathbb{Q}$ is at least $r$ for all but finitely many $t \in \mathbb{Q}$. Moreover, it is conjectured that $r_t \leq r+2$, except for a set of density $0$. In this article, when $E/\mathbb{Q}(T)$ has a torsion point of order $2$, under an assumption on the discriminant of a Weierstrass equation for $E/\mathbb{Q}(T)$, we produce an upper bound for $r_t$ that is valid for infinitely many $t$. We also present two examples of non-isotrivial elliptic curves $E/\mathbb{Q}(T)$ such that $r_t \leq r+1$ for infinitely many $t$.
\end{abstract}

\section{introduction}

Let $\pi : \mathcal{E} \rightarrow \mathbb{P}^1$ be a non-constant elliptic surface defined over $\mathbb{Q}$. By this we mean a two dimensional projective variety $\mathcal{E}$ endowed with a morphism $\pi$ as above such that all but finitely many fibers of $\pi$ are curves of genus one and such that there exists a section $\sigma$ to $\pi$. Let $E/\mathbb{Q}(T)$ be the generic fiber of $\mathcal{E}$, where $T$ is a coordinate of $\mathbb{P}_{\mathbb{Q}}^1$. The Mordell-Weil theorem for function fields (see \cite[Page 230]{silverman2}) asserts that $E(\mathbb{Q}(T))$ is a finitely generated group.

Denote by $r$ the rank of $E/\mathbb{Q}(T)$ of $\mathcal{E}$ and by $r_t$ the rank of the specialization $E_t/\mathbb{Q}$ of $\mathcal{E}$ at $T=t$, provided that $E_t/\mathbb{Q}$ is an elliptic curve. It follows from a theorem of Silverman (see \cite[Theorem C]{silvermanspecialization} or \cite[Theorem III.11.4]{silverman2}) that $r \leq r_t$ for all but finitely many $t$. A natural question to wonder is how far the above inequality is from being an equality. Assume from now on that $\pi : \mathcal{E} \rightarrow \mathbb{P}_{\mathbb{Q}}^1$ is non-isotrivial. Let $$\mathcal{N}(E)=\{ t \in \mathbb{P}^1(\mathbb{Q}) \: : \: E_t/\mathbb{Q} \text{ is an elliptic curve and } r_t=r \}$$ and 
$$\mathcal{F}(E)=\{ t \in \mathbb{P}^1(\mathbb{Q}) \: : \: E_t/\mathbb{Q} \text{ is an elliptic curve and } r_t \geq 2+r \}.$$
The density conjecture (see \cite[Page 556]{silvermandivisibility} or
\cite[Appendix A]{conradconradhelfgott}) predicts that $\mathcal{N}(\mathcal{E})$ is infinite while $\mathcal{F}(\mathcal{E})$ has density zero. 

Proving either of these two statements at the moment seems to be out of reach. Moreover, not a single (unconditional) example of a non-isotrivial elliptic surface for which $\mathcal{N}(E)$ is infinite is known. On the other hand, conditional examples relying on standard conjectures in analytic number theory have been found. For instance, under the assumption that there are infinitely many Mersenne primes, Caro and Pasten in \cite{caropasten} found an elliptic curve $E/\mathbb{Q}(T)$ of rank $0$ and infinitely many primes $q$ such that $E_q/\mathbb{Q}$ has rank $0$ as well. Moreover, work of Neuman and Setzer (see \cite{setzer} and \cite{neumann}) on elliptic curves with prime conductor combined with a conjecture of Bouniakowsky \cite{bouniakowsky} provides another such example.

For every $i \geq 1$ we let $$\mathcal{I}_i(E)=\{ t \in \mathbb{P}^1(\mathbb{Q}) \: : \: E_t \text{ is an elliptic curve and } r_t \leq r+1 \}.$$ In this article we are interested in providing examples of elliptic surfaces and explicit positive integers $i$ such that $\mathcal{I}_i(E)$ is infinite. Our first result is the following theorem (see Theorem \ref{thm1stfamily} and Theorem \ref{thm2ndfamily} below).

\begin{theorem}\label{thm2}
    Let $E/\mathbb{Q}(T)$ be either the elliptic curve given by the Weierstrass equation $$y^2=x^3+Tx^2-x$$ or the elliptic curve given by the Weierstrass equation $$y^2+xy=x^3+\frac{T-1}{4}x^2-x.$$ Then there exist infinitely many integers $n$ such that $\mathrm{rk } \:  E_n(\mathbb{Q}) \leq 1. $
    In particular, the set $\mathcal{I}_1(E)$ is infinite.
\end{theorem}

Before we state our next theorem we need to introduce some notation. If $F(x)$ is an irreducible polynomial with integer coefficients, then we write $\rho_F(p)$ for the number of solutions of the congruence $$F(x) \equiv 0 \: ( \text{mod } p).$$

\begin{theorem}\label{theorem1}
    Let $E/\mathbb{Q}(T)$ be a non-isotrivial elliptic curve whose Mordell-Weil group contains a point of order $2$. Fix an integral Weierstrass equation for $E/\mathbb{Q}(T)$ and denote by $\Delta(T) \in \mathbb{Z}[T]$ its discriminant.
    \begin{enumerate}
          \item Assume that $\Delta(T)=p_1^{a_1} \cdots p_m^{a_m} f(T)^k$ where $m \geq 0$,  $k, a_1,...,a_m>0$, and $f(T)$ is an irreducible polynomial with integral coefficients such that $\rho_{f}(p)<p$ for every prime $p$. Then there exist infinitely many positive integers $n$ such that $$\mathrm{rk } \: E_n(\mathbb{Q}) \leq 2\deg(\Delta)+2m+1. $$ In particular, the set $\mathcal{I}_{2\deg(\Delta)+2m+1}(E)$ is infinite.
         \item Assume that $\Delta(T)=p_1^{a_1} \cdots p_m^{a_m}T^{k_1}f(T)^{k_2}$, where $m \geq 0$, $k_1,k_2,a_1,...,a_m>0$, and $f(T) \neq \pm T$ is an irreducible polynomial with $\rho_{f}(p)<p$ for every prime $p$. Assume also that  $\rho_{f}(p)<p-1$ for every prime $p \leq \deg (f)+1$ such that $p \nmid f(0)$. Then there exist infinitely many positive integers $n$ such that $$\mathrm{rk } \:  E_n(\mathbb{Q}) \leq 4\deg(f)+2m+2. $$ In particular, the set $\mathcal{I}_{4\deg(f)+2m+2}(E)$ is infinite.
    \end{enumerate}
\end{theorem}

In fact, we prove a more general theorem where we also treat the general case where $\Delta(T)$ factors into a product of any number of irreducible polynomials. To prove our results we combine a bound on ranks of elliptic curves over $\mathbb{Q}$ that depends on their discriminants coming from $2$-descent (see \cite{alrp} and \cite{caropasten}) with results on almost prime values of polynomials that are derived from sieve methods in analytic number theory (see e.g. \cite{operadecribro} or \cite{sievemethodsbook}).

\section{Proofs of Theorems \ref{thm2} and \ref{theorem1}}

In this section we prove Theorems \ref{thm2} and \ref{theorem1}. In fact we will prove more general versions of the theorems stated in the introduction. Before we begin our proofs let us recall a theorem that provides an upper bound for the rank of elliptic curves with a torsion point of order $2$.

\begin{theorem}\label{thmdescent}(see \cite[Proposition 1.1]{alrp} and \cite[Theorem 2.3]{caropasten})
    Let $E/\mathbb{Q}$ be an elliptic curve that has a point of order $2$. 
    \begin{enumerate}
        \item If $E/\mathbb{Q}$ has an integral Weierstrass equation of the form $$y^2=x^3+Ax^2+Bx,$$ then $$\mathrm{rk} \: E(\mathbb{Q}) \leq \nu(A^2-4B)+\nu(B)-1,$$ where $\nu(n)$ be the number of positive prime divisors of a non-zero integer $n$.
        \item Let $\alpha$ and $\mu$ be the number of places of additive and of multiplicative reduction of $E/\mathbb{Q}$, respectively. Then $$\mathrm{rk} \: E(\mathbb{Q}) \leq 2\alpha+\mu-1.$$
    \end{enumerate}

\end{theorem}

\begin{remark}
    Elliptic curves for which the inequality of Part $(i)$ of Theorem \ref{thmdescent} is an equality are called elliptic curves of maximal Mordell-Weil rank. Examples of such curves have been exhibited by Aguirre, Lozano-Robledo, and Peral in \cite{alrp}.
\end{remark}

 Throughout the rest of this section we will denote by $P_r$ the set of positive integers with at most $r$ prime divisors, counted with multiplicity. We are now ready to proceed to the proof of Theorem \ref{thm2} for one of the two elliptic curves.

\begin{theorem}\label{thm1stfamily}
    Consider the elliptic curve $E/\mathbb{Q}(T)$ given by $$y^2=x^3+Tx^2-x.$$ Then there exist infinitely many integers $n$ such that $E_n/\mathbb{Q}$ has Mordell-Weil rank at most $1$. Moreover, there exists a positive constant $C$ such that if $X$ is sufficiently large, then $$\#\{n \: : \:  n \leq X \text{ and } \: \mathrm{rk } \: E_n(\mathbb{Q}) \leq 1 \} \geq C \frac{X}{\log X}.$$ 
\end{theorem}
\begin{proof}
    For every integer $n$ consider the elliptic curve $E_n/\mathbb{Q}$ given by $$y^2=x^3+nx^2-x$$ We first show that there exist infinitely many $n$ such that $E_n/\mathbb{Q}$ has Mordell-Weil rank at most $1$. Since $E/\mathbb{Q}(T)$ has a torsion point of order $2$ and the torsion subgroup of $E(\mathbb{Q}(T))$ injects in $E_n(\mathbb{Q})$ when $E_n/\mathbb{Q}$ is non-singular, we find that $E_n/\mathbb{Q}$ has a point of order $2$ for all but finitely many $n$. Therefore, it follows from Part $(i)$ of Theorem \ref{thmdescent} that $$\mathrm{rk} \: E_n(\mathbb{Q}) \leq \nu(n^2+4)+\nu(-1)-1=\nu(n^2+4)-1,$$ where $\nu(N)$ be the number of positive prime divisors of a non-zero integer $N$. 
    
    Therefore, if we can find infinitely many $n$ such that $n^2+4\in P_2$, then we are done. On the other hand, since for every prime $p$ we have that $\rho_{n^2+4}(p) \leq 2$ and $\rho_{n^2+4}(2)=1$, we get that $$\Gamma_{n^2+4}= \prod_{p \text{ prime}} \frac{1-\rho_{n^2+4}(p)/p}{1 -1/p}>0.$$ Therefore, it follows from \cite{lemkeoliver} that there exist infinitely many positive integers $n$ such that $n^2+4$ has at most two prime divisors, counted with multiplicity. This proves that there exist infinitely many positive integers $n$ such that $E_n/\mathbb{Q}$ has Mordell-Weil rank at most $1$.

    We now show the inequality of the theorem. Consider the polynomial $$h(n')=(n'+1)^2+4=n'^2+2n'+5.$$ According to \cite[Theorem 1]{kapoor} (see also \cite[Page 172]{iwaniec} and \cite{lemkeoliver}), if $X$ is sufficient large, we have that   $$\# \{n' \: : \:   n' \leq X \text{ and } \: h(n') \in P_2 \} \geq    \frac{1}{144}\prod_{p \text{ prime}} \frac{1-\rho_{h}(p)/p}{1 -1/p}\frac{X}{\log X}.$$ On the other hand, we see that $$\# \{n \: : \:  n \leq X \text{ and } \: n^2+4 \in P_2 \} \geq \# \{n' \: : \:  n' \leq X \text{ and } \: (n'+1)^2+4 \in P_2 \}-1.$$ 
    Therefore, we find that $$\# \{n \: : \:   n \leq X \text{ and } \: n^2+4 \in P_2 \} \geq    \frac{1}{144}\prod_{p \text{ prime}} \frac{1-\rho_{h}(p)/p}{1 -1/p}\frac{X}{\log X} -1 \geq C \frac{X}{\log X},$$ for all $X$ sufficiently large (picking an appropriate constant $C$). This proves our theorem.
\end{proof}

    Consider now the elliptic curve $E/\mathbb{Q}(T)$ given by the Weierstrass equation $$y^2+xy=x^3+\frac{T-1}{4}x^2-x.$$ Specializations of this curve have been studied by Neumann \cite{neumann} and Setzer \cite{setzer}. More precisely, it is proved in \cite[Theorem 2]{setzer} that if $p \neq 2,3,17$ is a prime  and $E/\mathbb{Q}$ is an elliptic curve of conductor $p$ with a torsion point of order $2$, then $p=b^2+64$ for some integer $b \equiv 1 \: (\text{mod } 4)$. In this case $E/\mathbb{Q}$ is isomorphic to either the curve $E_b/\mathbb{Q}$ or to a curve which is isogenous to $E_b/\mathbb{Q}$. 
    
    It follows from Part $(ii)$ of Theorem \ref{thmdescent} that if $p$ is a prime of the form $p=b^2+64$ for some integer $b \equiv 1 \: (\text{mod } 4)$, then the specialization $E_b/\mathbb{Q}$ has rank equal to $0$. According to a conjecture of Bouniakowsky \cite[Page 328]{bouniakowsky} there are infinitely many such numbers $b$. Without relying on any conjectures, we show below that we can find infinitely many integers $n$ such that $n^2+64 \in P_2$, which forces the rank of the corresponding curve $E_n/\mathbb{Q}$ to be at most $1$.

    \begin{theorem}\label{thm2ndfamily}
        Consider the elliptic curve $E/\mathbb{Q}(T)$ given by $$y^2+xy=x^3+\frac{T-1}{4}x^2-x.$$ Then there exist infinitely many integers $n$ such that $E_n/\mathbb{Q}$ has Mordell-Weil rank at most $1$.
    \end{theorem}
    \begin{proof}
        The proof is similar to the proof of Theorem \ref{thm1stfamily}. For every $n\in \mathbb{Z}$ consider the elliptic curve $E_n/\mathbb{Q}$ given by $$y^2+xy=x^3+\frac{n-1}{4}x^2-x.$$ The discriminant of $E_n/\mathbb{Q}$ is $$\Delta(n)=n^2+64$$ and the $c_4$-invariant is $$c_4(n)=n^2+48.$$ Since $E/\mathbb{Q}(T)$ has a torsion point of order $2$, we find that $E_n/\mathbb{Q}$ has a point of order $2$ for all but finitely many $E_n/\mathbb{Q}$. The strategy that we will follow for the rest of the proof is to try to control the primes of bad reduction of $E_n/\mathbb{Q}$ for sufficiently many integers $n$ and apply Theorem \ref{thmdescent}.

        \begin{claim}
            If $n^2+64 \in P_2$, then $n$ is odd.
        \end{claim}
        \begin{proof}[Proof of the claim]
            If $n^2+64$ is odd, then we must have that $n$ is odd. Therefore, assume that $$n^2+64=2q,$$ where $q$ is a prime (not necessarily distinct from $2$). This means that $n$ is even. If we write $n=2n'$, then $$2q=n^2+64=4n'^2+64,$$ then we see that $q$ is even so $q=2$. But then  $n^2+64=4,$ which is impossible. Therefore, $n$ must be odd.
        \end{proof}
If $n^2+64 \in P_2$, then $n^2+64$ is odd and, hence, $n$ is odd. Moreover, by replacing $n$ with $-n$ if necessary, we can also arrange that $n \equiv 1\: (\text{mod } 4).$ Thus the given Weierstrass equation for $E_n/\mathbb{Q}$ is an integral Weierstrass equation and by looking at the corresponding $c_4$-invariant we see that when $n^2+64 \in P_2$ every divisor of $\Delta(n)$ does not divide $c_4(n)$. This proves that, if $n^2+64 \in P_2$, then the curve $E_n/\mathbb{Q}$ has at most two primes of multiplicative reduction and no primes of additive reduction. Therefore, it follows from Part $(ii)$ of Theorem \ref{thmdescent} that $$\mathrm{rk} \: E_n(\mathbb{Q}) \leq 2 \cdot 0 +\mu_n-1=\mu_n-1,$$ where $\mu_n \leq 2$ is the number of primes of multiplicative reduction of $E_n/\mathbb{Q}$. This shows that when $n^2+64 \in P_2$, we have that $\mathrm{rk} \: E_n(\mathbb{Q}) \leq 1$.

Since for every prime $p$ we have that $\rho_{n^2+64}(p) \leq 2$ and $\rho_{n^2+64}(2)=1$, we get that $$\Gamma_{n^2+64}= \prod_{p \text{ prime}} \frac{1-\rho_{n^2+64}(p)/p}{1 -1/p}>0.$$ Therefore, it follows from \cite{lemkeoliver} that there exist infinitely many $n$ such that $n^2+64 \in P_2$.  Hence, there exist infinitely many integers $n$ such that $\mathrm{rk} \: E_n(\mathbb{Q}) \leq 1$.
    \end{proof}

Before we proceed to the proof of a slightly more general version of Theorem \ref{theorem1}, we need to recall two theorems on almost-prime values of polynomials that are derived from analytic number theory and will be needed in our proofs.

\begin{theorem}( \cite[Theorem 6]{richert} and \cite[Theorem 7]{richert})\label{thmdivisors}
    Let $F(x)$ be an irreducible polynomial of degree $g \geq  1$  with integral coefficients. Assume that $\rho_F(p)<p$ for every prime $p$.
    \begin{enumerate}
        \item Then there exists a constant $X_0(F)$ that depends on
        $F$ such that for every $X \geq X_0(F)$ we have that
        $$\#\{n \: : \: 1 \leq n \leq X, \: F(n) \in P_{g+1} \} \geq \frac{2}{3} \prod_p \frac{1-\rho_{F}(p)/p}{1 -1/p} \frac{X}{\log(X)}.$$ In particular, there exist infinitely many integers $n$ such that $F(n)$ has at most $g+1$ prime factors.
        \item Assume in addition that $\rho_{F}(p)<p-1$ for every prime $p \leq \deg (F)+1$ with $p \nmid F(0)$. Then there exist positive constants $C(F)$ and $X_0(F)$ that depend on
        $F$ such that for every $X \geq X_0(F)$ we have that $$\#\{ p \text{ prime } \: : \: 1 \leq p \leq X, \: F(p) \in P_{2g+1} \} \geq C(F) \frac{X}{\log^2(X)}.$$ In particular, there exist infinitely many prime numbers $p$ such that $F(p)$ has at most $2g+1$ prime factors.
    \end{enumerate}
\end{theorem}

\begin{theorem}\label{sievesbookthm}(\cite[Theorem 10.4]{sievemethodsbook})
   Let $F_1(x), F_2(x),...,F_g(x) $ be distinct irreducible polynomials with integral coefficients and write $F(x)=F_1(x)F_2(x) \cdots F_g(x)$ for their product. Assume that $\rho_F(p)<p$ for every prime $p$. Then there exists a positive integer $s$ that can be explicitly computed and depends on $F$ and a positive constant $C(F)$ that depends on $F$ such that for all $X$ sufficiently large we have that $$\#\{n \: : \: 1 \leq n \leq X, \: F(n) \in P_s \} \geq C(F)\frac{X}{\log^g(X)}.$$
\end{theorem}

We are now ready to proceed to the proof of a slightly more general version of Theorem \ref{theorem1}.

\begin{theorem}\label{thmlowerbounds}
    Let $E/\mathbb{Q}(T)$ be a non-isotrivial elliptic curve whose Mordell-Weil group contains a point of order $2$. Fix an integral Weierstrass equation for $E/\mathbb{Q}(T)$ and denote by $\Delta(T) \in \mathbb{Z}[T]$ its discriminant.
    \begin{enumerate}
          \item Assume that $\Delta(T)=p_1^{a_1} \cdots p_m^{a_m} f(T)^k$ where $m \geq 0$,  $k, a_1,...,a_m>0$, and $f(T)$ is an irreducible polynomial with integral coefficients such that $\rho_{f}(p)<p$ for every prime $p$. Then there exists a constant $X_0(f)$ that depends only on $f$ such that for every $X \geq X_0(f)$ we have that
         $$\#\{n \: : \: 1 \leq n \leq X, \: \mathrm{rk } \: E_n(\mathbb{Q}) \leq 2\deg(\Delta)+2m+1 \} \geq \frac{2}{3} \prod_p \frac{1-\rho_{f}(p)/p}{1 -1/p} \frac{X}{\log(X)}.$$ 
         In particular, there exist infinitely many positive integers $n$ such that $$\mathrm{rk }  E_n(\mathbb{Q}) \leq 2\deg(\Delta)+2m+1. $$
         \item Assume that $\Delta(T)=p_1^{a_1} \cdots p_m^{a_m}T^{k_1}f(T)^{k_2}$, where $m \geq 0$, $k_1,k_2,a_1,...,a_m>0$, and $f(T) \neq \pm T$ is an irreducible polynomial with $\rho_{f}(p)<p$ for every prime $p$. Assume also that  $\rho_{f}(p)<p-1$ for every prime $p \leq \deg (f)+1$ with $p \nmid f(0)$. Then there exist constants $C(f)$ and $X_1(f)$ that depend only on $f$ such that for every $X \geq X_1(f)$ we have that
         $$\#\{n \: : \: 1 \leq n \leq X, \: \mathrm{rk } \: E_n(\mathbb{Q}) \leq 4\deg (f)+2m+2 \} \geq C(f)\frac{X}{\log^2(X)}.$$ 
         In particular, there exist infinitely many positive integers $n$ such that 
         $$\mathrm{rk }  E_n(\mathbb{Q}) \leq 4\deg (f)+2m+2. $$
          
         \item More generally, write $\Delta(T)=p_1^{a_1} \cdots p_m^{a_m} f_1(T)^{h_1} \cdots f_g(T)^{h_g}$, where $p_1,...,p_m$  are distinct primes for some $m \geq 0$ and $f_1(T),...,f_g(T)$ are distinct irreducible polynomials with integral coefficients for some $g \geq 1$. Assume that $\rho_{f_1(T) \cdots f_g(T)}(p)<p$ for every prime $p$. Then there exists a positive constant $C(f_1,...,f_g)$ that depends on $f_1,...,f_g$ such that for all $X$ sufficiently large we have that
         $$\#\{n \: : \: 1 \leq n \leq X, \: \mathrm{rk } \: E_n(\mathbb{Q}) \leq 2(m+s) -1 \} \geq C(f_1,...,f_g) \frac{X}{\log^g(X)},$$ where $s$ is a positive integer that can be explicitly computed and depends on $\deg \Delta(T)$ and $g$.

    \end{enumerate}
\end{theorem}
\begin{proof}
    {\it Proof of Part $(i)$:} Let $E/\mathbb{Q}(T)$ be an elliptic curve as in $(i)$. Since $E/\mathbb{Q}(T)$ has a torsion point of order $2$ and the torsion subgroup of $E(\mathbb{Q}(T))$ injects in $E_t(\mathbb{Q})$ when $E_t/\mathbb{Q}$ is non-singular, we find that $E_t/\mathbb{Q}$ contains a point of order $2$ for all but finitely many $E_t/\mathbb{Q}$. On the other hand, applying Part $(i)$ of Theorem \ref{thmdivisors} to the polynomial $f(T)$ we find that there exists a constant $X_0(f)$ that depends on $f$ such that for every $X \geq X_0(f)$ we have that $$\#\{n \: : \: 1 \leq n \leq X, \: f(n) \in P_{\deg (f)+1} \} \geq \frac{2}{3} \prod_p \frac{1-\rho_{f}(p)/p}{1 -1/p} \frac{X}{\log(X)}.$$
    Recall that $\Delta(T)=p_1^{a_1} \cdots p_m^{a_m} f(T)^k$ and note that for each $n \in \mathbb{N}$ the primes of bad reduction of $E_n/\mathbb{Q}$ are a subset of the primes that divide $\Delta(n)$. Therefore, applying Part $(ii)$ of Theorem \ref{thmdescent} to each $E_n/\mathbb{Q}$ we find that if $f(n) \in P_{\deg (f)+1}$, then the rank of $E_n/\mathbb{Q}$ is at most $$2(\deg(f)+1+m)-1=2\deg(f)+2m+1.$$ Finally, since $\deg (f)=\deg (\Delta)$ our proof is complete.

    {\it Proof of Part $(ii)$:} Let now $E/\mathbb{Q}(T)$ be an elliptic curve as in $(ii)$. Similarly as in the previous part, we find that $E_t/\mathbb{Q}$ contains a point of order $2$ for all but finitely many $E_t/\mathbb{Q}$. On the other hand, applying Part $(ii)$ of Theorem \ref{thmdivisors} to the polynomial $f(T)$ we find that there exist positive constants $C(f)$ and $X_0(f)$ that depend on $f$ such that for every $X \geq X_0(f)$ we have that $$\#\{ p \text{ prime } \: : \: 1 \leq p \leq X, \: f(p) \in P_{2 \deg (f)+1} \} \geq C(f) \frac{X}{\log^2(X)}.$$ Recall $\Delta(T)=p_1^{a_1} \cdots p_m^{a_m}T^{k_1}f(T)^{k_2}$ and note that for each $n \in \mathbb{N}$ the primes of bad reduction of $E_n/\mathbb{Q}$ are a subset of the primes that divide $\Delta(n)$. Therefore, applying Part $(ii)$ of Theorem \ref{thmdescent} to each $E_n/\mathbb{Q}$ we find that if $n$ is prime with $f(n) \in P_{2 \deg (f)+1}$ the rank of $E_n/\mathbb{Q}$ is at most $$2(2 \deg (f) +2+m)-1=4\deg (f)+2m+2.$$

    {\it Proof of Part $(iii)$:}  Let now $E/\mathbb{Q}(T)$ be an elliptic curve as in $(iii)$. Similarly as in the previous parts, we find that $E_t/\mathbb{Q}$ contains a point of order $2$ for all but finitely many $E_t/\mathbb{Q}$. On the other hand, applying Theorem \ref{sievesbookthm} to the polynomial $f_1(T) \cdots f_g(T)$ we find that there exists a positive integer $s$ that can be explicitly computed and depends on $f_1(T),...,f_g(T)$ and a positive constant $C(f_1,...,f_g)$ that depends on $f_1(T),...,f_g(T)$ such that for all $X$ sufficiently large we have that $$\#\{n \: : \: 1 \leq n \leq X, \: f_1(n) \cdots f_g(T) \in P_s \} \geq C(F)\frac{X}{\log^g(X)}.$$ Recall that $\Delta(T)=p_1^{a_1} \cdots p_m^{a_m} f_1(T)^{h_1} \cdots f_g(T)^{h_g}$ and note that for each $n \in \mathbb{N}$ the primes of bad reduction of $E_n/\mathbb{Q}$ are a subset of the primes that divide $\Delta(n)$. Therefore, applying Part $(ii)$ of Theorem \ref{thmdescent} to each $E_n/\mathbb{Q}$ we find that if $n$ is prime with $f(n) \in P_{2 \deg (f)+1}$ the rank of $E_n/\mathbb{Q}$ is at most $$2(s+m)-1=2s+2m-1.$$ This completes the proof of our theorem.
     
\end{proof}

We end this article by presenting some examples where Theorem \ref{thmlowerbounds} can be applied to find explicit bounds for the ranks of infinitely many specializations.

\begin{example}
    Consider the elliptic curve $E/\mathbb{Q}(T)$ given by the equation $$y^2=x^3+g(T)x^2-\lambda x,$$ where $g(T) \in \mathbb{Z}[T]$ and $\lambda$ is a positive integer. The above equation has discriminant $$\Delta(T)=16(-\lambda)^2(g(T)^2+4\lambda).$$ When $g(T)^2+4\lambda $ is an irreducible polynomial such that $\rho_g(p)<p$ for every prime $p$, Part $(i)$ of Theorem \ref{thmlowerbounds} shows that there exist infinitely many positive integers $n$ such that $$\mathrm{rk }  E_n(\mathbb{Q}) \leq 4\deg(g)+2v(\lambda)+3, $$ where $v(\lambda)$ is the number of positive prime divisors of $\lambda$.
\end{example}

\begin{example}
    Consider the elliptic curve $E/\mathbb{Q}(T)$ given by the equation $$y^2+xy-Ty=x^3-Tx^2,$$ which has discriminant $$\Delta(T)=T^4(1+16T).$$ It is well known (see \cite[Section 4.4]{hus}) that if $E'/\mathbb{Q}$ is any elliptic curve with a torsion point of order $4$, then there exists $\lambda \in \mathbb{Q}$ such that $E'/\mathbb{Q}$ is isomorphic to $E_{\lambda}/\mathbb{Q}$. According to Part $(ii)$ of Theorem \ref{thmlowerbounds} we see that there exist infinitely many integers $n$ such that the rank of $E_n/\mathbb{Q}$ is at most $6$. The reader can easily check that the discriminant of $E/\mathbb{Q}(T)$ satisfies the hypotheses of of Part $(ii)$ of Theorem \ref{thmlowerbounds}.
\end{example}

\begin{remark}\label{rmks}
    Keep the same notation as in the Theorem \ref{thmlowerbounds}. We note that given an elliptic curve that satisfies the conditions of Part $(iii)$ of Theorem \ref{thmlowerbounds} then the corresponding number $s$ can be explicitly computed. A formula for $s$ can be found on \cite[Page 283]{sievemethodsbook}. In the simple case where $g=2$, i.e., when $$\Delta(T)=p_1^{a_1} \cdots p_m^{a_m}f_1(T)^{h_1} f_2(T)^{h_2},$$ then $s$ can be computed based only on the degree of $\Delta(T).$ For example, when $\deg (\Delta(T))$ is equal to $3,4,5,$ or $6$, then $s$ is equal to $7,9,10,$ or $11$, respectively (see \cite[Page 287]{sievemethodsbook}).
\end{remark}

\begin{example}
    Consider the elliptic curve $E/\mathbb{Q}(T)$ given by the equation $$y^2=x^3+(T+1)x^2-(T^2+1)x$$ which has discriminant $$\Delta(T)=16(T^2+1)^2((T+1)^2+4(T^2+1))=16(T^2+1)^2(5T^2+2T+5).$$ According to  Part $(iii)$ of Theorem \ref{thmlowerbounds} combined with Remark \ref{rmks} we find that there exist infinitely many $n$ such that the rank of $E_n/\mathbb{Q}$ is at most $9$.
\end{example}

\bibliographystyle{plain}
\bibliography{bibliography.bib}

\end{document}